\documentclass[11pt,a4paper,reqno]{article}
\usepackage{amsmath}
\usepackage{amsthm}
\usepackage{enumerate}
\usepackage{amssymb}

\usepackage{verbatim}
\usepackage{url}
\usepackage[latin1]{inputenc}

\usepackage{caption}
\usepackage{graphicx,color}
\usepackage{tikz}

\newcommand\ett{\mathbf1}

\def\N{{\mathbb N}}
\def\Z{{\mathbb Z}}

\def\R{{\mathbb R}}

\def\Pr{{\mathbb P}}
\def\Ex{{\mathbb E}}

\def\Ac{{\mathcal A}}
\def\Bc{{\mathcal B}}
\def\Cc{{\mathcal C}}
\def\Dc{{\mathcal D}}

\def\Fc{{\mathcal F}}

\def\Uc{{\mathcal U}}

\def\hZ{{\hat{Z}_\x}}
\def\hW{{\hat{W}}}

\def\x{\mathbf{x}}

\def\0{\mathbf{0}}

\newcommand{\vertiii}[1]{{\left\vert\kern-0.25ex\left\vert\kern-0.25ex\left\vert #1  \right\vert\kern-0.25ex\right\vert\kern-0.25ex\right\vert}}

\def\be{\begin{equation}}
\def\ee{\end{equation}}
\def\bea{\begin{equation*}}
\def\eea{\end{equation*}}
\def\bal{\begin{aligned}}
\def\eal{\end{aligned}}

\DeclareMathOperator{\ind}{{\bf 1}}

\def\le{\leqslant}

\def\ge{\geqslant}

\newtheorem{thm}{Theorem}[section]
\newtheorem{lma}[thm]{Lemma}
\newtheorem{cor}[thm]{Corollary}
\newtheorem{prop}[thm]{Proposition}

\newtheorem{claim}[thm]{Claim}
\newtheorem{conj}[thm]{Conjecture}
\newtheorem{prob}[thm]{Problem}

\theoremstyle{remark}
\newtheorem*{remark}{Remark}
\newtheorem{preex}[thm]{Example}

\theoremstyle{definition}

\begin{document}

\title{Competition in growth and urns\thanks{This work was in part supported
	by grant 637-2013-7302 from the Swedish Research Council (DA), the Knut
	and Alice Wallenberg Foundation (SJ) and by CNPq (Proc.~303275/2013-8), FAPERJ (Proc.~201.598/2014), and ERC Starting Grant 680275 MALIG (RM). Part of the work was done during visits of SJ and RM to the Isaac Newton Institute for Mathematical Sciences, 
EPSCR Grant Number EP/K032208/1.
}}
\date{\today}
\author{Daniel Ahlberg, Simon Griffiths, Svante Janson, and Robert Morris}

\maketitle

\begin{abstract}
We study survival among two competing types in two settings: a planar growth model related to two-neighbour bootstrap percolation, and a system of urns with graph-based interactions. In the planar growth model, uncoloured sites are given a colour at rate $0$, $1$ or $\infty$, depending on whether they have zero, one, or at least two neighbours of that colour. In the urn scheme, each vertex of a graph $G$ has an associated urn containing some number of either blue or red balls (but not both). At each time step, a ball is chosen uniformly at random from all those currently present in the system, a ball of the same colour is added to each neighbouring urn, and balls in the same urn but of different colours annihilate on a one-for-one basis. We show that, for every connected graph $G$ and every initial configuration, only one colour survives almost surely. As a corollary, we deduce that in the two-type growth model on $\Z^2$, one of the colours only infects a finite number of sites with probability one. We also discuss generalisations to higher dimensions and multi-type processes, and list a number of open problems and conjectures.
\end{abstract}

\section{Introduction}

A model of competition for space between two or more growing entities was introduced in the context of first-passage percolation on $\Z^2$ by H\"aggstr\"om and Pemantle~\cite{hagpem98}. Beyond the mere beauty of the problem, the authors were motivated by the closely-related problem of understanding the geodesic structure in the random metric on $\Z^2$ induced by the first-passage percolation configuration. More precisely, they realized that positive probability of unbounded growth of two different types implies a lower bound on the number of semi-infinite geodesics originating from a given point, making the first rigorous progress on a question raised by Newman~\cite{newman95}. H\"aggstr\"om and Pemantle showed that if two growing entities on $\Z^2$ infect unoccupied neighbours at rate 1, which corresponds to first-passage percolation with exponential weights, with positive probability both grow indefinitely. 

Since the original work of H\"aggstr\"om and Pemantle~\cite{hagpem98}, competing growth has attracted considerable attention in a variety of settings. In later work H\"aggstr\"om and Pemantle~\cite{hagpem00} considered competition between two growing entities on $\Z^2$ with different growth rates. Garet and Marchand~\cite{garmar05} and Hoffman~\cite{hoffman05} have considered the analogous problem in higher dimensions for first-passage percolation with general edge weights. Unbounded initial configurations were examined by Deijfen and H\"aggstr\"om~\cite{deihag07}, while competition with more that two growing entities was studied by Hoffman~\cite{hoffman08} and Damron and Hochman~\cite{damhoc13}. See the excellent survey by Deijfen and H\"aggstr\"om~\cite{deihag08} for a general introduction to the area.

A striking feature of these lattice models is that (in all known cases)
coexistence occurs in the case of equal strength competitors. Moreover, on a
finite torus both colours occupy (with high probability) a positive
proportion of the vertices, even when the growth is highly biased,  
 see Antunovi\'c, Dekel, Mossel and Peres~\cite{antdekmosper}. In the setting of random graphs the outcome is sometimes different: there are settings in which one of the competing entities on a random graph ends up with a $1-o(1)$ proportion of the vertices with high probability, see for example Deijfen and van der Hofstad~\cite{deihof16} and Antunovi\'c, Dekel, Mossel and Peres~\cite{antdekmosper}. On the other hand,  Antunovi\'c, Mossel and Racz~\cite{antmosrac} showed that a certain model of preferential attachment gives rise to random graphs in which both entities end up with a positive proportion of the vertices. 

In this paper we will consider the problem of coexistence in a growth model on $\Z^2$, partly motivated by a process studied by Kesten and Schonmann~\cite{kessch95}, in which vertices of the same colour cooperate, and infect any common neighbour instantly. We will show that in this model, for any finite initial configuration with two colours, there is only one surviving type. In order to prove this theorem, we will naturally be led to study a model of competition in a system of urns that interact via some graph $G$. We will show that, for any finite system of urns, and for any initial configuration with two colours, only one type survives. We will also discuss the growth model in higher dimensions and the multicoloured setting. In a follow-up paper~\cite{ahlgrijanmor2} we shall study infinite systems of urns. In each of these cases, there are a number of natural open problems and conjectures.

\subsection{Competition in growth}\label{sec:growth}

Motivated by the Glauber dynamics of the Ising model at low temperatures, Kesten and Schonmann~\cite{kessch95} introduced the following process on $\Z^d$. Let $A \subset \Z^d$ be a set of initially `infected' sites, and let healthy sites become infected at rate $0$, $1$ or $\lambda$, depending on whether they have zero, one, or at least two already-infected neighbours. (Infected sites stay infected forever, and $\lambda \ge 1$ is a parameter of the model.) In~\cite{kessch95}, the authors determined the rate of growth (up to a constant factor) and the asymptotic shape of the infected droplet in the limit $\lambda \to \infty$. We remark that the case $\lambda = 1$ corresponds to a close relative of first-passage percolation, in which sites instead of bonds are assigned random weights, known as the Eden model~\cite{eden61}. A more complicated `nucleation and growth' model, in which sites can become infected despite having no infected neighbours, was studied in~\cite{dehsch97a,cerman13a,bolgrimorrolsmi}, and applied to the Ising model in~\cite{dehsch97b,cerman13b}. 

We shall be interested in a two-colour competitive variant of the Kesten--Schonmann process on $\Z^2$ (see Sections~\ref{sec:multicolours} and~\ref{sec:dimensions} for a discussion of the problem with more colours and/or in higher dimensions). Let $R$ and $B$ be (disjoint) sets of red and blue (respectively) vertices of $\Z^2$ at time $t = 0$, and let not-yet-coloured sites be coloured red (resp. blue) at rate $0$, $1$ or $\infty$, depending on whether they have zero, one, or at least two red (resp. blue) neighbours.\footnote{So that the model is well-defined, we will assume that no site in $\Z^2 \setminus (R \cup B)$ has more than one neighbour in either set. Note that for any such pair $(R,B)$, almost surely no vertex will be coloured with more than one colour.} We will refer to this as the \emph{two-type growth model on $\Z^2$}. 

We will prove the following theorem, which states that coexistence does \emph{not} occur in this model in two dimensions. Let us say that an initial configuration is `finite' if only a finite number of sites are initially coloured, and that a colour `survives' if an infinite number of sites are eventually given that colour.

\begin{thm}\label{thm:growth}
For any finite initial configuration, the two-type growth model on $\Z^2$ almost surely has only one surviving colour. 
\end{thm}

Note that moreover, for any finite initial configuration in which one colour does not already `block' the other, each colour has a positive probability of being the surviving species. We believe that Theorem~\ref{thm:growth} can be generalised to higher dimensions and to more than two types (see Sections~\ref{sec:multicolours} and~\ref{sec:dimensions}), but we do not know how to prove either statement. We also expect there to be only one surviving colour in the natural `biased' version of the two-type growth model, in which the colours have different rates (so a vertex with one neighbour of colour $j$ is coloured $j$ at rate $\lambda_j$), but (perhaps surprisingly) our methods also do not seem to easily extend to this setting. On the other hand, we conjecture that for any finite $\lambda \ge 1$, if two-neighbour infections occur at rate $\lambda$ instead of $\infty$, then for any initial configuration with at least two non-blocked colours, with positive probability more than one colour survives.

\subsection{Competition in urns}

We will deduce Theorem~\ref{thm:growth} from a corresponding statement about competition in urns on a cycle. The deduction is very simple, so let us begin by outlining it (a detailed proof is given in Section~\ref{sec:proof}). We will associate each set of (coloured) vertices $X \subset \Z^2$ with a (coloured) subset of $\R^2$, by placing on each vertex $x \in X$ a unit square of the same colour as $x$. Now, given any finite initial configuration, after some time $t$ the subset of~$\R^2$ associated with the coloured vertices will form a (solid) connected component $S \subset \R^2$. Let $\partial S$ denote the boundary of $S$, and suppose that $\partial S$ consists of $k$ intervals (that is, connected subsets of $\partial S$) of each colour. Note that, since two-neighbour infections are instantaneous, $\partial S$ consists of $4k + 4$ monochromatic (horizontal or vertical) line segments (some of which may be empty), see Figure~\ref{fig:evolution}. We associate an urn with each of these $4k + 4$ lines (which we will refer to as the `sides' of $S$), and define a (cyclic) graph on the urns by adding an edge between the urns associated with neighbouring sides. We place $\ell$ red (resp. blue) balls in an urn if the corresponding side is red (resp. blue) and has 
length $\ell$.

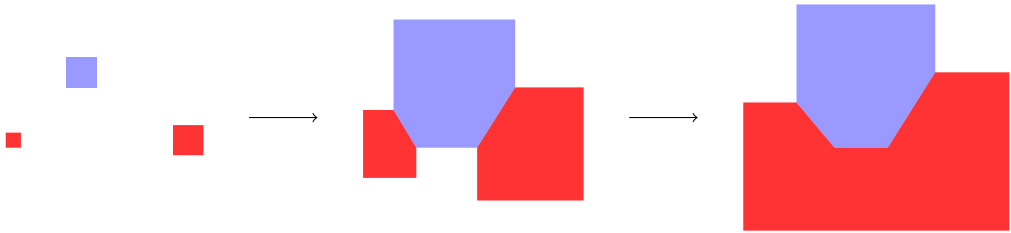
\begin{figure}[htbp]
\begin{center}
\begin{tikzpicture}[scale=.4]
\fill[fill=red, opacity=.8]
	(-.25,0) rectangle (.25,.5)
	(5.25,-.25) rectangle (6.25,.75);
\fill[fill=blue, opacity=.4]
	(1.75,2) rectangle (2.75,3);
\fill[fill=red, opacity=.8]
	(11.5,-1) -- (13.25,-1) -- (13.25,0) -- (12.5,1.25) -- (11.5,1.25) -- cycle
	(15.25,-1.75) -- (15.25,0) -- (16.5,2) -- (18.75,2) -- (18.75,-1.75) -- cycle;
\fill[fill=blue, opacity=.4]
	(12.5,1.25) -- (12.5,4.25) -- (16.5,4.25) -- (16.5,2) -- (15.25,0) -- (13.25,0) -- cycle;
\fill[fill=red, opacity=.8]
	(24,-2.75) -- (24,1.5) -- (25.75,1.5) -- (27,0) -- (28.75,0) -- (30.3125,2.5) -- (32.75,2.5) -- (32.75,-2.75) --cycle;
\fill[fill=blue, opacity=.4]
	(25.75,1.5) -- (25.75,4.75) -- (30.3125,4.75) -- (30.3125,2.5) -- (28.75,0) -- (27,0) -- cycle;
\draw [->] (7.75,1) -- (10,1);
\draw [->] (20.25,1) -- (22.5,1);
\end{tikzpicture}
\end{center}
\caption{The initial growing phase. In the second figure the coloured sites have formed a connected component whose perimeter consists of $k$ red
intervals and $k$ blue intervals, where $k = 2$, 
divided into a total of 12 vertical and horizontal line segments. In the third figure $k=1$ and there are 8 line segments.}
\label{fig:evolution}
\end{figure}

Consider the next site to be coloured after time $t$; it is the endpoint (not in $S$) of a uniformly chosen element of the edge-boundary of the set of coloured vertices. The effect of this vertex being coloured is that the corresponding side of $S$ shifts (instantly) sideways by distance one, and thus causes the lengths of the neighbouring sides to increase or decrease by one, depending on their colour. The corresponding process on the urns is as follows: at each (discrete) time step, we choose a ball uniformly at random (from the set of all balls in the urns), and add a ball of the same colour to each of the neighbouring urns; if there are balls of a different colour in some of these urns then the new ball immediately annihilates with one of them. In order to show that only one colour survives in the two-type growth process, it will therefore suffice to prove that only one colour survives in the urn process on a cycle. 

We will in fact prove the following vast generalisation of this claim. Given a collection $V$ of urns, and a graph $G$ with vertex set $V$, let us call the process described above the \emph{two-type urn process} (or \emph{two-type urn scheme}) on $G$. Here a colour is said to `survive' if it is present in the system at arbitrarily large times.

\begin{thm}\label{thm:urns}
For every finite, connected graph $G$, and any finite initial configuration, the two-type urn process on $G$ almost surely has only one surviving colour. 
\end{thm}

We will show moreover that the number of balls of the surviving colour in the urn on vertex $v$ is almost surely asymptotically proportional to the corresponding entry of the Perron-Frobenius eigenvector of $G$, see Section~\ref{sec:competing:urns}. For background on urn schemes, see for example the surveys of Bena\"im~\cite{benaim99} and Pemantle~\cite{pemantle07}.

\subsection{Competition between many types}\label{sec:multicolours}

For each $s \in \N$, we can define the \emph{$s$-type growth model on $\Z^2$} simply by allowing $s$ different colours. (More precisely, for each colour $i \in [s]$, not-yet-coloured sites are given colour $i$ at rate $0$, $1$ or $\infty$, depending on whether they have zero, one, or at least two neighbours coloured $i$.) We strongly believe that the conclusion of Theorem~\ref{thm:growth} should hold in this more general setting, but we do not see an easy way to prove it using our methods.

\begin{conj}\label{conj:rtype:growth}
For every $s \ge 2$, and every finite initial configuration, the $s$-type growth model on $\Z^2$ almost surely has only one surviving colour. 
\end{conj}

The main reason our proof of Theorem~\ref{thm:growth} does not generalise to larger values of $s$ is that the corresponding statement for urns is false in general. Given a graph $G$, define the \emph{$s$-type urn process on $G$} by allowing balls of $s$ different colours (though each urn must, of course, still be monochromatic). As before, at each step we choose a ball uniformly at random and add a ball of the same colour to each of the neighbouring urns, with annihilation on a one-to-one basis between balls of different colours. 

\begin{prop}\label{prop:counterexample}
For every $s \ge 3$ there exists a finite connected graph $G$ and an initial configuration such that with positive probability all $s$ colours survive in the $s$-type urn process on $G$. 
\end{prop}

The construction we will use to prove Proposition~\ref{prop:counterexample} is extremely simple: we take $G$ to be a collection of $s$ triangles all sharing a common vertex\footnote{We call this the ``Risk graph" because, as in the boardgame, the players who are behind invariably gang up on the leader, and as a result no-one ever wins.}, place one ball of colour $i$ on each vertex that is only in triangle $i$, and no balls on the central vertex, see Figure~\ref{fig:risk}.\medskip

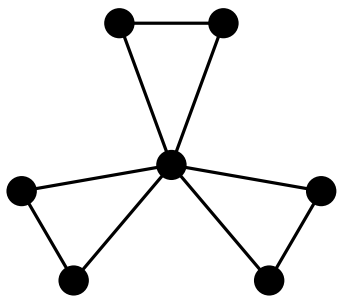
\begin{figure}[htbp]
\begin{center}
\begin{tikzpicture}[scale=1]
\def\m{3}
\def\r{2}
\def\Acngles{{70,110,190,230,310,350}} 
\foreach \i in{0,...,5}{
	\draw[very thick] (0,0) -- ({\r*cos(\Acngles[\i])},{\r*sin(\Acngles[\i])});
	}
\foreach \i in{0,...,2}{
	\pgfmathparse{2*\i+1} \let\ii\pgfmathresult
	\draw[very thick] ({\r*cos(\Acngles[2*\i])},{\r*sin(\Acngles[2*\i])}) -- ({\r*cos(\Acngles[\ii])},{\r*sin(\Acngles[\ii])});
	}
\foreach \i in{0,...,5}{
	\fill[fill=black] ({\r*cos(\Acngles[\i])},{\r*sin(\Acngles[\i])}) circle (0.2);
	}
\fill[fill=black] (0,0) circle (0.2); 
\end{tikzpicture}
\end{center}
\caption{The Risk graph on $s=3$ triangles.}
\label{fig:risk}
\end{figure}

In order to prove Conjecture~\ref{conj:rtype:growth}, however, it would be sufficient to prove an $s$-type analogue of Theorem~\ref{thm:urns} in the special case of cycles. We do not know how to prove this either, but strongly believe it to be true.

\begin{conj}\label{conj:rtype:urns}
For every $s,k \ge 3$, and every finite initial configuration, the $s$-type urn model on $C_k$, the cycle of length $k$, almost surely has only one surviving colour. 
\end{conj}

On the other hand, in order to prove an $s$-type analogue of Theorem~\ref{thm:growth} in the half-plane it is sufficient to prove an $s$-type analogue of Theorem~\ref{thm:urns} for paths, which follows easily from Theorem~\ref{thm:urns} by relabeling monochromatic stretches of colours in an alternating fashion, see Section~\ref{sec:proof}. This leads to the following corollary.

\begin{cor}\label{cor:half-plane}
For every $s\ge1$, and every finite initial configuration, the $s$-type growth model on $\Z^2$ restricted to a half-plane almost surely has only one surviving colour. 
\end{cor}

It is easy to see that a similar result holds in (for example) the quarter-plane.

\subsection{Growth in higher dimensions}\label{sec:dimensions}

Motivated by its applications to the study of the Ising model in $d$-dimensions, Cerf and Manzo~\cite{cerman13a,cerman13b} studied a model of nucleation and growth on $\Z^d$ in which the rate of infection increases rapidly with the number of already-infected neighbours. In this section we will discuss competition in a simpler version of their model in which the infection rates are either $0$, $1$ or $\infty$. To be precise, let us define the \emph{$r$-neighbour $s$-type growth model on $\Z^d$} as follows: not-yet-coloured sites are given colour $i$ at rate $0$ if they have no coloured neighbours, at rate $1$ if they have $j \in \{1,\ldots,r-1\}$ neighbours of colour $i$, and at rate $\infty$, if they have at least $r$ neighbours of colour $i$. We make the following conjecture.

\begin{conj}\label{conj:growth:highdim}
Let $d,r,s \ge 2$. The $r$-neighbour $s$-type growth model on $\Z^d$ almost surely has only one surviving colour for any finite initial configuration if and only if $r = 2$. 
\end{conj}

To see why (we believe that) multiple colours should survive whenever $r \ge 3$, observe that the sides of a droplet in $\Z^d$ grow according to a process that (roughly speaking) resembles $(r-1)$-neighbour bootstrap percolation on $[n]^{d-1}$. By a classical theorem of Aizenman and Lebowitz~\cite{AL}, it follows that a droplet of sidelength $n$ will grow in each direction at speed at most  
$\Theta\big( \log n \big)^{d-1}$ as long as $d \ge r \ge 3$. 
Since this function grows only poly-logarithmically (rather than polynomially) with $n$, one would not expect the colour that is ``winning" at a given (sufficiently large) time to maintain its lead, in the sense that the ratio of side-lengths of the droplets should converge to $1$. 
On the other hand, if $r = 2$ then a droplet of side-length $n$ will grow in each direction at speed $\Theta\big( n^{d-1} \big)$, and as a result any colour that gains a slight advantage should increase her lead as time goes on, and eventually swallow each of the other colours. 
Despite this simple and natural heuristic, it seems likely that some fairly daunting technical difficulties would have to be overcome in order to give a rigorous proof.

\subsection{A word on our methods}

A variety of techniques have been developed over the years to study urn schemes and similar processes. These techniques include exchangeability, branching processes and stochastic approximation algorithms, see~\cite{pemantle07} for an overview. While each of these sets of techniques have their individual strengths, we have found that an approach based on branching process theory best fits our purposes. Indeed, for monochromatic initial configurations the classical theory of (multi-type) branching processes reveals a great deal of information about our urn scheme; we shall review the basics of this theory in Section~\ref{sec:monochromatic}.

For non-monochromatic initial configurations the situation is more complicated, due to the interaction (via annihilation) of balls of different colours, which depends on the spatial dependence induced by the underlying graph. In order to deal with this difficulty we shall define (see Section~\ref{sec:competing:urns}) a `conservative' system that encodes our two-type urn scheme. By studying this system we shall be able to prove Theorem~\ref{thm:urns}, and hence deduce Theorem~\ref{thm:growth} and Corollary~\ref{cor:half-plane} (see Section~\ref{sec:proof}). In Section~\ref{sec:counter-example} we will sketch a proof of coexistence on the Risk graph. Finally, in Section~\ref{sec:open}, we list several further open problems.

\section{The monochromatic urn scheme}\label{sec:monochromatic}

In this section we will prove a basic (but useful) lemma about the 1-type urn process on a graph $G$ with $n$ vertices. In order to simplify the analysis, we will embed the (discrete-time) urn process into continuous time by equipping each ball in the system with an independent Poisson clock; when a clock rings, the corresponding ball immediately sends a copy of itself along each edge of the graph incident to its vertex (we will call this a \emph{nucleation}), and each new ball obtains its own clock that is independent of all others. We will denote the initial configuration by $\x \in \N^n$, where $\N = \{0,1,2,\ldots\}$, and the configuration at time $t$ (starting from $\x$) by $Z_\x(t)$.


Since there is no interaction between particles in the monochromatic system, the resulting process $(Z_\x(t))_{t\ge0}$ is a multi-type continuous time Markov branching process with $n$ different types; see \cite[Chapter~III]{athney72} or \cite[Section~6.1]{jagers75}. Loosely speaking, a continuous time multi-type Markov branching process is a continuous time Markov process $\{ X(t) : t \ge 0 \}$ taking values in $\N^n$, for some $n \ge 1$, in which $X(t) = \big( X_1(t),\ldots,X_n(t) \big)$ is interpreted as the number of particles of $n$ different types, and each particle of type $i$ (independently of other particles in the system) has an exponentially distributed lifetime, at the end of which it produces a finite collection of new particles according to some distribution on $\N^n$.

In the case we are concerned with here, all types have the same (unit)
expected lifetime, at the end of which a particle of type $i$ produces a
copy of itself, along with a particle of each type $j$ for which $(i,j)$ is
an edge in $G$. The offspring distribution for this branching process is
encoded in the adjacency matrix of $G$. Let $\lambda(G)$ denote the
Perron--Frobenius eigenvalue of (the adjacency matrix of) $G$, and let
$\pi(G)$ be the corresponding eigenvector scaled so that $\|\pi(G)\|_2 =
1$. We recall from Perron--Frobenius theory that for any finite connected
graph $G$ the eigenvalue $\lambda(G)$ is real and (strictly) positive and
$\pi(G)$ has (strictly) positive entries.

From the classical work of Athreya~\cite{athreya68} we obtain the following result.

\begin{prop}\label{prop:onetype:urn}
Let $G$ be a finite, connected graph on $n$ vertices and let $\x \in \N^n$, $\x \not\equiv 0$. There exists a random variable $W = W(G,\x)$ such that $\Ex[W]=\x\cdot\pi$, $\Pr( W = 0 ) = 0$ and
$$
\lim_{t \to \infty} e^{-\lambda t} Z_\x(t) \, = \, W \pi 
$$
almost surely and in $L^2$, where $\lambda = \lambda(G)$ and $\pi = \pi(G)$.  
\end{prop}

\begin{proof}
The work of Athreya immediately implies the almost sure~\cite[Theorem~1]{athreya68} and $L^2$~\cite[Theorem~3]{athreya68} existence of the limit. The core of the argument is based on the fact that the family $\{e^{-\lambda t} Z_\x(t)\cdot\pi:t\ge0\}$ is a martingale, which implies in particular that 
\begin{equation}\label{eq:Wmean}
\Ex[W(G,\x)] \, = \, \lim_{t\to\infty}e^{-\lambda t} \, \Ex[Z_\x(t)\cdot\pi] \, = \, \Ex[Z_\x(0)\cdot\pi] \, = \, \x\cdot\pi \, > \, 0.
\end{equation}
The limit $W=W(G,\x)$ is further known to satisfy a variety of properties. For instance, for $W$ to be absolutely continuous on $[0,\infty)$, it is sufficient that $Z_\x(t)$ cannot attain the all zero state, see~\cite[Theorem~2]{athreya68}. However, instead of relying on the general theory, it will be instructive (for what is to come in Section~\ref{sec:competing:urns}) to prove that $\Pr( W = 0 ) = 0$ by hand.

To do so, let us fix an initial state $\x \not\equiv 0$, and consider the first nucleation to occur -- we will track the balls (and their descendants) created in this first nucleation separately. Indeed, let $Y_\x(t) \in \N^n$ denote the vector of these balls and their descendants at time $t$, and let $\hZ(t) = Z_\x(t) - Y_\x(t)$ denote the vector of the remaining balls. Note that $\hZ(t)$ evolves as a 1-type urn process on $G$, with initial state $\x$, but delayed by time $\tau_1$ (the random time of the first nucleation).

Athreya's result implies that $(\hZ(t))_{t\ge0}$ and $(Y_\x(t))_{t\ge0}$,
once rescaled by $e^{-\lambda t}$, converge almost surely. 
That is, if we denote their limits by $\hW$ and $W_Y$, then almost surely
$$
W\pi \, = \, \lim_{t \to \infty} e^{-\lambda t} Z_\x(t) \, = \, \lim_{t \to \infty} e^{-\lambda t} \big( \hZ(t) + Y_\x(t)\big) \, = \, \big( \hW + W_Y \big) \pi.
$$
Moreover, by~\eqref{eq:Wmean} we have $\Ex[W_Y] > 0$, and since $(\hZ(t +
\tau_1))_{t \ge 0}$ has the same distribution as $(Z_\x(t))_{t \ge 0}$, 
and is independent of $\tau_1$, 
it follows that $\hW$ equals $e^{-\lambda \tau_1}W$ in distribution.

We claim that 
\begin{equation}\label{eq:Wzero}
\Pr( W = 0 ) \, = \, \Pr( \hW = W_Y = 0 ) \, = \, \Pr( W = 0 ) \, \Pr( W_Y = 0 ).
\end{equation}
The first step follows since $W = \hW + W_Y$, and $\hW$ and $W_Y$ are both supported on $[0,\infty)$, while the second follows since the events $\{\hW = 0\}$ and $\{W_Y = 0\}$ are independent, and because $\Pr( \hW = 0 ) = \Pr( W = 0 )$. It follows that either $\Pr( W = 0 ) = 0$ or $\Pr( W_Y = 0 ) = 1$, but the latter is contradicted by the fact that $\Ex[W_Y] > 0$, so we are done.
\end{proof}

\section{The competing urn scheme}\label{sec:competing:urns}

In this section we will study the two-type urn process, embedded into
continuous time, as in the previous section. We will denote the initial
configuration by $\x \in \Z^n$, where $\x(v)$ is equal to the number of red
balls at vertex $v$ minus the number of blue balls at $v$, and the
configuration at time $t$ (starting from $\x$) by $Z_\x(t)$. The vector
$Z_\x(t)$ thus denotes the difference between the number of red and blue
balls in each urn at time $t$. 
Since each urn contains only balls of one colour, this means that the number
of balls at vertex $v$ is $|Z_\x(t)(v)|$, and that they are red [blue] if
$Z_\x(t)(v)>0$ [$<0$].

The main result of this section is the
following strengthening of Theorem~\ref{thm:urns}, which generalises
Proposition~\ref{prop:onetype:urn}. 

\begin{thm}\label{thm:Z-limit}
Let $G$ be a finite connected graph on $n$ vertices and let $\x \in \Z^n$, $\x\not\equiv0$. There exists a random variable $W = W(G,\x)$ such that $\Pr( W = 0 ) = 0$ and
$$
\lim_{t \to \infty} e^{-\lambda t} Z_\x(t) \, = \, W \pi 
$$
almost surely and in $L^2$, where $\lambda = \lambda(G)$ and $\pi = \pi(G)$.
\end{thm}

To see why Theorem~\ref{thm:urns} follows from Theorem~\ref{thm:Z-limit} it
suffices to recall that for any finite connected graph $G$ the eigenvector
$\pi(G)$ has strictly positive entries; 
thus the entries in $W\pi$ are either all positive or all negative.

The key idea in the proof is to consider a conservative system from which we will be able to read out the evolution of the annihilating system. This will be achieved as follows: when a red and a blue ball are supposed to annihilate one another, we will instead `merge' them to form a purple ball
[\tikz \fill[red, opacity=.8] (.9ex,.9ex) circle (.9ex);
+ \tikz \fill[blue, opacity=.4] (.9ex,.9ex) circle (.9ex);
= \tikz \fill[purple] (.9ex,.9ex) circle (.9ex);].
The purple balls will continue to evolve according to their own Poisson clocks, producing (purple) descendants, but purple balls do not interact with each other, or with any of the other balls in the system.

The harder part of the proof of Theorem~\ref{thm:Z-limit} will be to show that $\Pr( W = 0 ) = 0$. We begin with the easier part: showing that the limit exists.

\begin{lma}\label{lma:Z-limit}
Let $G$ be a finite connected graph on $n$ vertices and let $\x \in \Z^n$, $\x\not\equiv0$. There exists a random variable $W = W(G,\x)$ such that
$$
\lim_{t \to \infty} e^{-\lambda t} Z_\x(t) \, = \, W \pi
$$
almost surely and in $L^2$. 
\end{lma}

\begin{proof}
If $\x \in \N^n$ then the conclusion follows from Proposition~\ref{prop:onetype:urn}, so we may assume that $W$ exists in the monochromatic setting. We will deduce the result in the general setting by considering three monochromatic systems, and applying this result twice.

To be precise, let $R_\x(t)$, $B_\x(t)$ and $P_\x(t)$ denote the vectors indicating the number of red, blue and purple balls at each site at time $t$ in the modified (conservative) process described above, in which red and blue balls are merged into purple rather than annihilated. The following observation is trivial. 


\begin{claim}
$Z_\x(t) = R_\x(t) - B_\x(t)$.
\end{claim}

\begin{proof}[Proof of claim] 
This follows simply because purple balls do not interact with either red or
blue balls, so the evolution of the red and blue balls  
is the same in both the original and conservative processes. 
\end{proof} 


We next apply Proposition~\ref{prop:onetype:urn} to the red and blue balls separately. 

\begin{claim}
There exist random variables $W^+$ and $W^-$ such that, 
almost surely and in $L^2$, we have
$$
\lim_{t \to \infty} e^{-\lambda t} \big( R_\x(t)+ P_\x(t) \big) \, = \, W^+ \pi \quad \text{and} \quad \lim_{t \to \infty} e^{-\lambda t} \big( B_\x(t)+ P_\x(t) \big) \, = \, W^- \pi.
$$
\end{claim}

\begin{proof}[Proof of claim] 
Note that each purple ball corresponds to a red ball that would have been present in the system if no blue balls had been initially present. In other words, $R_\x(t) + P_\x(t)$ has the same distribution as the number of red balls starting from the configuration $\x^+$, where $\x^+_i = \max\{x_i,0\}$. Since $\x^+$ is monochromatic, the first limit exists by Proposition~\ref{prop:onetype:urn}, and the second follows by an identical argument. 
\end{proof}


Note that the random variables $W^+$ and $W^-$ are not independent. Nevertheless, it follows from Claims~1 and~2 that
$$
\lim_{t \to \infty} e^{-\lambda t} Z_\x(t) \, = \, \lim_{t \to \infty} e^{-\lambda t} \Big( \big( R_\x(t)+ P_\x(t) \big) - \big( B_\x(t)+ P_\x(t) \big) \Big) \, = \, \big( W^+ - W^- \big) \pi,
$$
almost surely and in $L^2$, as required.
\end{proof}

For the rest of the section, let us write $W = W(G,\x)$ for the random variable given by Lemma~\ref{lma:Z-limit}. In order to complete the proof of Theorem~\ref{thm:Z-limit}, it remains to prove that $W = W^+ - W^-$ is almost surely nonzero. Note that, while it follows from Proposition~\ref{prop:onetype:urn} that $W^+$ and $W^-$ are both almost surely nonzero, since they are not independent we cannot use this fact directly to deduce the same for annihilating systems. 

Indeed, to prove that $W$ is almost surely nonzero we will again use the conservative urn scheme (in which red and blue balls are merged into purple rather than annihilated), but we will also add an additional mark to some of the balls (and their offspring). The red, blue and purple balls will interact as before and independently of the additional marks, while the marks will evolve as follows:
\begin{itemize}
\item[$(a)$] Let $b$ be the first ball to nucleate, and let $v$ be the corresponding vertex of $G$. We add a mark to each of the $d(v)$ new balls created in the first nucleation. 
\item[$(b)$] We will assume (without loss of generality) that $b$ is blue,
  in which case no red balls will ever be marked. 
(If $b$ is red, interchange red and blue below.)   
\item[$(c)$] When a marked ball nucleates, its offspring will also be marked.
\item[$(d)$] When a marked blue ball merges with a red ball, it produces a marked purple ball.
\item[$(e)$] Marks prefer blue to purple. Therefore, if an unmarked blue ball and a marked purple ball find themselves in the same urn, then the mark will immediately jump from (one of) the purple ball(s) to (one of) the blue ball(s).\footnote{Alternatively, but equivalently, red balls prefer to be paired with unmarked blue as opposed to marked blue balls, and thus regroup when they find themselves in the same urn as an unmarked blue ball. To put it another way, the marked balls can be thought of as `floating on top' of the other balls, sometimes being carried for a while by a red ball, but not influencing the configuration of unmarked balls. Due to exchangeability this description gives a process with the same law.}
\end{itemize}
As before, we will write $R_\x(t)$, $B_\x(t)$ and $P_\x(t)$ for the vectors indicating the number of red, blue and purple balls at each site at time $t$, starting from configuration $\x$. We will write $B^*_\x(t)$ (resp. $P^*_\x(t)$) for the vectors indicating the number of marked blue (resp. purple) balls, and $B^\circ_\x(t) = B_\x(t) - B^*_\x(t)$ (resp. $P^\circ_\x(t) = P_\x(t) - P^*_\x(t)$) for the unmarked balls. 



Together with Lemma~\ref{lma:Z-limit}, the following lemma completes the proof of Theorem~\ref{thm:Z-limit}. 

\begin{lma}
For every finite connected graph $G$ on $n$ vertices, and every $\x \in \Z^n$, $\x\not\equiv0$, 
$$\Pr\big( W(G,\x) = 0 \big) = 0.$$
\end{lma}

\begin{proof}
Suppose first that the first ball $b$ that nucleates is blue. 
Recall that $Z_\x(t) = R_\x(t) - B_\x(t)$ denotes the difference between the number of red and blue balls present (marks ignored) at each vertex at time $t$, and let 
$$M_\x(t) \, := \, B^*_\x(t) + P^*_\x(t)$$ 
denote the vector indicating the number of marked balls (blue plus purple) at each vertex at time $t$. Moreover, define
$$Y_\x(t) \, := \, R_\x(t) - B^\circ_\x(t) + P_\x^*(t),$$
and observe that
\begin{equation}
  \label{zym}
Z_\x(t) \, = \, Y_\x(t) - M_\x(t).
\end{equation}

If the ball $b$ is red, we interchange red and blue and change signs on 
$Y_\x(t)$ and $M_\x(t)$ above; then \eqref{zym} still holds.

We claim that both $e^{-\lambda t} Y_\x(t)$ and $e^{-\lambda t} M_\x(t)$ converge. 

\begin{claim}
There exist random variables $W^Y$ and $W^M$ such that
$$
\lim_{t \to \infty} e^{-\lambda t} Y_\x(t) \, = \, W^Y \pi \quad \text{and} \quad \lim_{t \to \infty} e^{-\lambda t} M_\x(t) \, = \, W^M \pi
$$
almost surely.
Moreover, 
$$W(G,\x) \stackrel{d}{=} e^{-\lambda\tau_1} W^Y \quad \text{and} \quad \Pr\big( W^M = 0 \big) = 0,$$
where $\tau_1$ is the (random) time of the first nucleation. 
\end{claim}

\begin{proof}[Proof of claim] 
The existence and claimed properties of $W^M$ follow from
Proposition~\ref{prop:onetype:urn}, since 
if $b$ is blue,  
$M_\x(t)$ has the same distribution as a monochromatic system starting (at
the random time of the first nucleation) with one blue ball at each
neighbour of $v$, and if $b$ is red, the same holds for $-M_\x(t)$. 
(Indeed, neither step $(d)$ nor step $(e)$ in the
definition above affect the value of $M_\x(t)$.) 
For $W^Y$, they follow by
Lemma~\ref{lma:Z-limit}, since $Y_\x(t)$ has the same distribution as the
two-type process started with $\x$ at time $\tau_1$. (Indeed, 
assuming that $b$ is blue, 
the only
difference is that some of the red balls have merged with marked blue balls,
and are therefore counted in $Y_\x(t)$ as marked purple balls. Note that
when a mark jumps from a purple to  a blue ball, it does not affect the
value of $Y_\x(t)$.) 
\end{proof} 

It follows that
\be\label{eq:joint-limit}
W\pi\,=\,\lim_{t\to\infty}e^{-\lambda t} Z_\x(t) \, = \, \lim_{t \to \infty} e^{-\lambda t} \big( Y_\x(t) - M_\x(t)\big) \, = \, \big( W^Y - W^M \big) \pi.
\ee

Since the convergence in~\eqref{eq:joint-limit} is almost sure, and since $W^M \ne 0$ with probability one, regardless of the initial configuration, we conclude that both $W$ and $W^Y$ cannot be zero simultaneously. Since $W \stackrel{d}{=} e^{-\lambda\tau_1} W^Y$, it follows that
\be\label{eq:W-bound}
\Pr\big(W(G,\x)=0\big) \, \le \, 1/2
\ee
for every initial configuration $\x \not\equiv 0$.

We would like to strengthen this to conclude that $\Pr(W=0)=0$. Let $\Fc_s$ denote the $\sigma$-algebra encoding the evolution of the process $(Z_\x(t))_{t\ge0}$ up to time $s$ and let $B=\{W(G,\x)=0\}$. Since $W \pi$ is the point-wise limit of $e^{-\lambda t}Z_\x(t)$, the event $B$ is contained in $\Fc_\infty:=\sigma\big(\bigcup_{s\ge0}\Fc_s \big)$. The sequence $\big( \Pr(B \,|\, \Fc_s) \big)_{s\ge0}$ is a bounded martingale, and therefore convergent. Moreover, Levy's 0--1 law implies that, almost surely,
$$
\Pr(B\,|\,\Fc_s) \, \to \, \Pr(B\,|\,\Fc_\infty) \, = \, \ind_B
$$
as $s\to\infty$. However,~\eqref{eq:W-bound} shows that, almost surely,
$$
\Pr( B \, | \, \Fc_s ) \, = \, \Pr\big( W(G,\x) = 0 \, \big| \, Z_\x(s) \big) \, = \, \Pr\big( W(G,Z_\x(s)) = 0 \, \big| \, Z_\x(s) \big) \, \le \, 1/2,
$$
so the event $B$ can only occur on a set of measure zero, as required.
\end{proof}

\section{Deducing Theorem~\ref{thm:growth} from Theorem~\ref{thm:urns}}\label{sec:proof}

In this section we will deduce Theorem~\ref{thm:growth} and Corollary~\ref{cor:half-plane} from Theorem~\ref{thm:urns}, by applying it to a cycle and a path, respectively. Since both deductions are straightforward (and were already outlined in the Introduction), we will be somewhat brief with the details. 

\begin{proof}[Proof of Theorem~\ref{thm:growth}]
Let $R,B \subset \Z^2$ be finite sets of red and blue sites such that each $x \in \Z^2 \setminus (R \cup B)$ has at most one neighbour in each set. Let $A_t$ denote the (random) set of coloured vertices at time~$t$ in the two-type growth model on $\Z^2$ starting from this configuration, and let $S_t$ denote the coloured subset of $\R^2$ obtained by placing on each vertex $x \in A_t$ a unit square of the same colour as $x$. 

Let $\tau_0 \ge 0$ denote the random time at which $S_t$ becomes connected, and note that $\tau_0$ is almost surely finite. Unless one colour at this point already surrounds the other (in which case we are done), the outer boundary $\partial S_t$ of $S_t$ consists of $2k$ (for some $k \ge 1$) monochromatic intervals (i.e., connected subsets of $\partial S_t$) alternating between red and blue. (Since any inner boundaries are irrelevant to the long-term development of the process, we will ignore them.) Observe that the number of such monochromatic intervals can (deterministically) never increase; we will show that it decreases in finite time almost surely. Since this holds for every $k \ge 1$, it will be sufficient to prove the theorem.

To prove this, suppose (for a contradiction) that the outer boundary of the coloured set consists of $2k$ alternating monochromatic intervals for all $t \ge \tau_0$. At time $\tau_0$, select a point on the outer boundary that marks a break between a red and a blue interval, and accompany its movement during the evolution of the process. Let $\tau_1<\tau_2<\ldots$ denote the subsequent times of nucleation. At time $\tau_j$, for $j\ge0$, follow the outer boundary of the coloured component counter-clockwise, starting from the selected point, and stop when reaching the next corner or another intersection between red and blue, whichever comes first. If the stretch just followed was red and had length $\ell$, then set $D_1(j)=\ell$. If it was blue and had length $\ell$, then set $D_1(j)=-\ell$. Assume that $D_i(j)$ has been defined. If $D_i(j)$ ended at a corner, then repeat the previous procedure to define $D_{i+1}(j)$. If $D_i(j)$ ended before reaching the next corner, but because an intersection between red and blue was reached, then let $D_{i+1}(j)=0$ and repeat the previous procedure to define $D_{i+2}(j)$. We stop once we have returned to our starting point.

\begin{claim}\label{claim:growth:1}
If the outer boundary $\partial S_t$ of the coloured set at time $t = \tau_j$ consists of $2k$ alternating monochromatic intervals, then exactly $4k+4$ variables $D_i(j)$ will be defined.
\end{claim}

\begin{proof}[Proof of claim]
At each time $t \ge \tau_0$, the outer boundary has $m$ inner corners and $m+4$ outer corners, for some $m \in \N$. A break between red and blue along the perimeter cannot occur at an outer corner, must occur at each inner corner, and can also occur along the sides. In particular, $m$ of the breaks occur at inner corners, and $2k - m$ of the breaks occur along line segments of the boundary. We obtain one variable for each corner (outer or inner), and two additional variables for each break along a side, and we therefore define exactly $2m + 4 + 2 (2k - m) = 4k + 4$ variables, as claimed.
\end{proof}

The claimed correspondence with the urn scheme on $C_{4k+4}$ now follows immediately. 

\begin{claim}\label{claim:growth:2}
If $\partial S_t$ consists of $2k$ alternating monochromatic intervals for every $t \ge \tau_0$, then $\big(D_1(j),\ldots,D_{4k+4}(j) \big)_{j \in \N}$ coincides with a two-type urn scheme on $C_{4k+4}$.
\end{claim}

\begin{proof}[Proof of claim]
At each nucleation, if the nucleating site corresponds to $D_i(j)$, then $D_{i-1}(j)$ and $D_{i+1}(j)$ will either increase or decrease by one, depending on the sign of $D_i(j)$.
\end{proof}

By Theorem~\ref{thm:urns}, the two-type urn scheme on $C_{4k+4}$ will almost surely become monochromatic in finite time. However, we assumed that the outer boundary of the coloured set consists of $2k$ alternating monochromatic intervals for all $t \ge \tau_0$, which would imply that the urn scheme also has $2k$ alternating monochromatic intervals for all $t \ge \tau_0$. This contradiction completes the proof of Theorem~\ref{thm:growth}.
\end{proof}

The proof of Corollary~\ref{cor:half-plane} is almost identical to that of Theorem~\ref{thm:growth}, so we will be somewhat more brief with the details. 

\begin{proof}[Proof of Corollary~\ref{cor:half-plane}]
As in the proof of Theorem~\ref{thm:growth}, let $A_t$ denote the set of coloured vertices at time~$t$ in the $s$-type growth model on a half-plane starting from some given finite configuration, and let $S_t$ denote the coloured subset of $\R^2$ obtained by placing on each vertex $x \in A_t$ a unit square of the same colour as $x$. 

Observe that there almost surely exists a finite time $\tau_0$ after which $S_t$ is connected and touches the boundary of the half-plane. The outer boundary $\partial S_t$ of $S_t$ (which does not include the boundary of the half-plane) consists of $k$ monochromatic intervals (for some $k \ge 1$); the number of such monochromatic intervals can (deterministically) never increase, and we will show that if $k \ge 2$ then it decreases in finite time almost surely. Since this holds for every $k \ge 2$, it will be sufficient to prove the corollary.

To prove this, suppose (for a contradiction) that the outer boundary of the coloured set consists of $k \ge 2$ monochromatic intervals for all $t \ge \tau_0$. We relabel these intervals (alternatingly) `red' and `blue', regardless of their previous colour, and construct the random variables $D_i(j)$ as above, starting and ending at the two points of $\partial S_t$ on the boundary of the half-plane. By the proofs of Claims~\ref{claim:growth:1} and~\ref{claim:growth:2}, it follows that $\big(D_1(j),\ldots,D_\ell(j) \big)_{j \in \N}$ coincides with a two-type urn scheme on $P_\ell$, the path of length $\ell - 1$, for some $\ell \ge 2$.

Now, by Theorem~\ref{thm:urns}, the two-type urn scheme on $P_\ell$ will almost surely become monochromatic in finite time. However, we assumed that the outer boundary of the coloured set consists of $k \ge 2$ alternating monochromatic intervals for all $t \ge \tau_0$, which would imply that the urn scheme also has $k$ alternating monochromatic intervals for all $t \ge \tau_0$. This contradiction completes the proof of Corollary~\ref{cor:half-plane}.
\end{proof}

\section{A counter-example for multi-coloured urns}\label{sec:counter-example}


In this section we will sketch the proof of Proposition~\ref{prop:counterexample}, for simplicity in the case $s = 3$. We leave the (straightforward) generalisation to an arbitrary $s \ge 3$ to the reader.

Let $R_3$ denote the graph with 7 vertices consisting of three triangles sharing a vertex (see Figure~\ref{fig:risk}). We call the vertex of degree $6$ the \emph{centre vertex} and denote it $v_0$, and the remaining vertices we call \emph{periphery vertices}.  Observe that the periphery consists of three \emph{periphery pairs}.  We consider the initial configuration in which each periphery pair is assigned a distinct colour (red, blue or green), and each vertex of the pair is assigned $n$ balls of that colour. We will show that, as $n \to \infty$, the probability that one of the colours fails to survive forever tends to zero. 
(It follows that for any $n \ge 1$, the probability that all colours
survive is positive. Indeed, for any $N \in \N$, with positive probability we reach the state with $N$ balls at each peripheral vertex without any nucleation at $v_0$.)

The intuition behind the proof is very simple: we expect the periphery to contain a large, growing, and roughly equal number of balls of each colour, and the centre to contain relatively few balls at all times. While this holds, it is stable: the periphery balls will behave like a simple Polya Urn (with three colours), whereas the centre vertex will have a strong drift against whichever colour finds itself there at a given time, since it is being attacked by (roughly) twice as many balls as are supporting it. Our task is to bound the probability that we ever leave this `metastable state'. 

To do so, we make the following definitions. First, for each $k \ge 0$, let $A_k$ denote the number of balls at the centre after $k$ steps (so $A_0 = 0$), and define a `bad' event 
$$\Ac_k \, = \, \big\{ A_k > k^{1/6}n^{1/6} \big\}.$$
Next, again for each $k \ge 0$, let $B_k$ denote the number of balls in the periphery after $k$ steps (so $B_0 = 6n$), and define a second bad event 
$$\Bc_k \, = \, \big\{ B_k < 3n + k/2 \big\}.$$
Next, for each $k \ge 0$, each periphery vertex $v$, and each colour $j$, define $r_k(v)$ (resp. $r_k(j)$) to be the proportion of periphery balls that are at vertex $v$ (resp. of colour $j$) after $k$ steps. We define two more bad events as follows:
$$\Cc_k \, = \, \big\{ r_k(v) = 0 \textup{ for some } v \neq v_0 \big\} \quad \textup{and} \quad \Dc_k \, = \, \big\{ r_k(j) < f(k) \textup{ for some colour } j \big\},$$
where $f(k) = 1/3 - \alpha - \sum_{i = n}^{n+k-1} i^{-3/2}$ for some small
constant $\alpha > 0$ 
  (assuming that $n$ is so large that $f(k)>0$). 
Define also
$$\Ac \, :=\, \bigcup_{k\ge 1} \bigg( \Ac_k \cap \bigcap_{\ell = 0}^{k-1} \big( \Bc_\ell \cup \Cc_\ell \cup \Dc_\ell \big)^c \bigg),$$
and define $\Bc$, $\Cc$ and $\Dc$ similarly. Note that if none of these events occurs then all three colours survive; it will therefore suffice to prove that each has probability $o(1)$ as $n \to \infty$. 

We will bound the probability of the events $\Ac,\Bc$ and $\Cc$ in terms of deviations of biased random walks. Indeed, observe first that if $k$ is minimal such that $\Ac_k$ holds, then a majority of the last $k^{1/6}n^{1/6}$ balls sent to the centre must have been of the same colour. Note that $f(k) \ge 3/10$ for every $k \ge 0$ (if $\alpha > 0$ is sufficiently small and $n$ is sufficiently large), and so if $\Dc_\ell$ does not hold then $r_k(j) \le 1 - 2f(k) \le 2/5$ for each colour $j$. Thus, if $\Dc_\ell$ does not hold for any $\ell < k$, this event has probability at most
$$3 \cdot 2^{k^{1/6}n^{1/6}} \cdot \bigg( \frac{2}{5} \bigg)^{k^{1/6}n^{1/6}/2} \cdot \bigg( \frac{3}{5} \bigg)^{k^{1/6}n^{1/6}/2} \le 3 \cdot \bigg( \frac{24}{25} \bigg)^{k^{1/6}n^{1/6}/2}.$$ 
Summing over $k \ge 1$, it follows that the event $\Ac$ occurs with probability $o(1)$.

The proofs that $\Pr(\Bc)=o(1)$ and $\Pr(\Cc)=o(1)$ are similar, so we shall
be even briefer. If $k$ is minimal such that $\Bc_k$ holds, and $\Ac_\ell
\cup \Cc_\ell$ does not hold for any $\ell < k$, 
then, for each $v$, 
 the expected increase in $B_k$ in each step (conditional on what has happened so far) is greater than $3/4$. The probability that $B_k$ increases by at most $k/2 - 3n$ in the first $k$ steps is therefore exponentially small in $n + k$, and (summing over $k$) is $o(1)$ as $n \to \infty$. Similarly, if $k$ is minimal such that $\Cc_k$ holds, and $\Ac_\ell \cup \Bc_\ell \cup \Dc_\ell$ does not hold for any $\ell < k$, then the expected increase in $r_k(v) B_k$ in each step (conditional on what has happened so far) is again greater than $3/4$. The probability that $r_k(v) B_k$ decreases by $n$ in the first $k$ steps is therefore exponentially small in $n + k$, and so we are done as before. 

For the event $\Dc$, we will use the Azuma--Hoeffding inequality. Let us fix a colour $j$ and define a stopping time $T = \min\{ t : \Ac_t\cup \Bc_t\cup \Cc_t\cup \Dc_t \textup{ occurs} \}$. We will bound the probability that $r_T(j) < f(T)$. To do so, let us define a sub-martingale $(Y_k)_{k \ge 0}$ by setting 
$$Y_k \, := \, r_k(j) - f(k)$$
if  $k \le T$, and $Y_k := Y_{k-1}$ otherwise. (To see that this is a sub-martingale, observe that if $k < T$ then the number of periphery balls of colour $j$ after step $k+1$ is always at least $r_k(j) B_k - 2$, is $r_k(j) B_k + 1$ with probability $r_k(j) - O(A_k/B_k)$, and is $r_k(j) B_k$ with probability $1 - r_k(j) - O(A_k/B_k)$. Via a short calculation, it follows that
$$\Ex\big[ Y_{k+1} - Y_k \,|\, Z_\x(k) \big] \ge \frac{1}{(n + k)^{3/2}} - O\bigg( \frac{A_k}{B_k^2} \bigg) > 0$$
if $k < T$, and is zero otherwise.)

We will apply the following well-known variant of the Azuma--Hoeffding inequality (see, e.g.,~\cite[Eq.~(3.30)]{mcdiarmid98}): if $(Y_k)_{k\ge 0}$ is a sub-martingale with $|Y_{k+1} - Y_k| \le c_k$ almost surely for every $k \ge 0$, then
$$\Pr\Big( \min_{j \le k} Y_j \, < \, Y_0 - t \Big) \, \le \, \exp\left(\frac{-t^2}{2\sum_{k \ge 0} c_k^2} \right)$$
for every $t \ge 0$. Noting that $Y_0 = \alpha$ and $|Y_{k+1} - Y_k| \le 8/(n+k)$ for every $k \ge 0$. Since the event $\Dc_k$ implies that $Y_k < 0$, it follows that
 $$\Pr(\Dc) \, \le\, \exp\left(\frac{-\alpha^2}{2^7\sum_{i=1}^{\infty}(n + i)^{-2}}\right) \, \le \, e^{-\alpha^3 n} = o(1)$$
as $n \to \infty$, as claimed. This completes the proof of the proposition. 


\section{A few more open problems}\label{sec:open}

In the Introduction we stated several problems and conjectures about the models studied in this paper; we would like to finish the paper by mentioning a few more. We begin by asking on which graphs coexistence is possible in the $s$-type urn process. 

\begin{prob}
For each $s \ge 3$, determine the family of finite connected graphs $G$ with the following property: for any finite initial configuration, the $s$-type urn process on $G$ almost surely has only one surviving colour. 
\end{prob}

When coexistence is not possible (for example, in the two-type growth model on $\Z^2$, or the two-type urn process), it would also be interesting to understand the expected time until only one colour remains, and the probability that a given colour is the one that survives. (The former question was suggested to us by Yuval Peres.) 

\begin{prob}
For the two-type growth model on $\Z^2$, or the two-type urn model on a finite connected graph $G$, and  for a given (finite) initial configuration, determine the rate of decay of the probability that more than one type survives until time~$t$.
\end{prob}

A related question asks for the typical `shape' of the colour that does not survive. 

\begin{prob}
Consider the two-type growth model on $\Z^2$, started with a single red and a single blue site at distance $n$ apart. At the time of extinction, what can be said about the shape of the surrounded colour? 
\end{prob}

In particular, one might guess that there exists a 1-parameter family of `shapes' such that the surrounded region (once suitably rescaled) converges to a member of this family as $n \to \infty$. 

Finally, it is natural to consider the two-type growth model with different update rules. For example, given a finite collection $\Uc = \{ X_1,\ldots,X_m \}$ of finite subsets of $\Z^2 \setminus \{ \0 \}$, we can define the \emph{$\Uc$-bootstrap growth model on $\Z^2$} as follows: a not-yet-coloured site $v$ is given colour $i$ at rate $0$ if it has no coloured neighbours, at rate $1$ if it has at least one neighbour of colour $i$, and at rate $\infty$ if every element of the set $v + X$ is coloured $i$ for some $X \in \Uc$. For example, if $\Uc$ consists of the $2$-element subsets of the four nearest neighbours of the origin, then we recover the two-type growth model defined in Section~\ref{sec:growth}. These very general update rules were introduced recently (in the context of bootstrap percolation) by Bollob\'as, Smith and Uzzell~\cite{BSU}, and studied further in~\cite{BBPS,BDMS}. 

\begin{prob}\label{prob:Ugrowth}
For which update families $\Uc$ does the $\Uc$-bootstrap growth model on $\Z^2$ almost surely have only one surviving colour for any finite initial configuration?
\end{prob}

Of course, one could generalize further (for example, to higher dimensions and to multiple types), or replace the 1-neighbour rule for colouring at rate $1$ with a different update family. However, we expect Problem~\ref{prob:Ugrowth} to already be rather challenging. 



\newcommand{\noopsort}[1]{}\def\cprime{$'$}

\medskip

{\small
\noindent
{\sc Daniel Ahlberg\\
Instituto Nacional de Matem\'atica Pura e Aplicada\\
Estrada Dona Castorina 110, 22460-320 Rio de Janeiro, Brasil\\
Department of Mathematics, Uppsala University\\
SE-75106 Uppsala, Sweden}\\
\url{http://w3.impa.br/~ahlberg}\\
\texttt{ahlberg@impa.br}\\

\noindent
{\sc Simon Griffiths\\
Department of Statistics, University of Oxford\\
1 South Parks Road, OX1 3TG Oxford, United Kingdom}\\
\texttt{griffith@stats.ox.ac.uk}\\

\noindent
{\sc Svante Janson\\
Department of Mathematics, Uppsala University\\
SE-75106 Uppsala, Sweden}\\
\url{http://www2.math.uu.se/~svante}\\
\texttt{svante.janson@math.uu.se}\\

\noindent
{\sc Robert Morris\\
Instituto Nacional de Matem\'atica Pura e Aplicada\\
Estrada Dona Castorina 110, 22460-320 Rio de Janeiro, Brasil}\\
\url{http://w3.impa.br/~rob}\\
\texttt{rob@impa.br}
}

\end{document}